\documentclass[12pt]{article}
\usepackage{amsfonts,amsmath,amssymb}
\usepackage{slashed,setspace}
\usepackage[all]{xy}
\usepackage{fullpage}
\usepackage{enumerate}
\usepackage[active]{srcltx}
\def\leftnote#1{\vadjust{\setbox1=\vtop{\hsize 30mm\parindent=0pt\bf\baselineskip=9pt\rightskip=4mm plus 4mm#1}\hbox{\kern-3cm\smash{\box1}}}}

\newtheorem{theorem}{Theorem}[section]
\newtheorem{definition}[theorem]{Definition}
\newtheorem{lemma}[theorem]{Lemma}
\newtheorem{remark}[theorem]{Remark}

\newenvironment{proof}[1][Proof]{\par\addvspace{2mm}\noindent\textbf{#1.} }{\ \rule{0.5em}{0.5em}\par\vspace{4mm}}
\newcommand{\Q}{\mathbb{Q}}
\newcommand{\bo}{{\frak O}}
\newcommand{\D}{{\frak D}}

\newcommand{\bp}{{\frak P}}

\newcommand{\K}{\mathcal{K}}

\newcommand{\N}{{\mathcal N}}

\DeclareMathOperator{\Gal}{Gal}

\begin{document}

\bibliographystyle{plain}

\title{Explicit Construction of Self-Dual Integral Normal Bases for the Square-Root of the Inverse Different}

\author{Erik Jarl Pickett}

\maketitle

\begin{abstract}

Let $K$ be a finite extension of $\Q_p$, let $L/K$ be a finite abelian Galois extension of odd degree 
and let $\bo_L$ be the valuation ring of $L$. We define $A_{L/K}$ to be the unique
fractional $\bo_L$-ideal with square equal to the inverse different
of $L/K$. For $p$ an odd prime and $L/\Q_p$ contained in certain cyclotomic extensions, 
Erez has described integral normal bases for $A_{L/\Q_p}$ 
that are self-dual with respect to the trace form. Assuming $K/\Q_p$ to be unramified we generate odd abelian weakly ramified extensions of $K$ using Lubin-Tate formal groups. 
We then use Dwork's exponential power series to explicitly construct self-dual integral normal bases for the square-root of the
inverse different in these extensions.

\end{abstract}

\section{Introduction} \label{Intro}

 Let $K$ be a finite extension of $\Q_p$ and let $\bo_K$ be the valuation ring of $K$ with unique maximal ideal $\bp_K$.
 We let $L/K$ be a finite
Galois extension of odd degree with Galois group $G$ and let $\bo_L$ be
the integral closure of $\bo_K$ in $L$. From \cite{Serre} IV \S2 Prop 4, this means that the different, $\D_{L/K}$, of $L/K$ will have an even valuation, and so we define $A_{L/K}$ to be the unique fractional ideal such that \[A_{L/K}=\D_{L/K}^{-1/2}.\]

We let $T_{L/K}:L\times L\rightarrow K$ be the symmetric
non-degenerate $K$-bilinear form associated to the trace map (i.e.,
$T_{L/K}(x,y)=Tr_{L/K}(xy)$) which is $G$-invariant in the sense
that $T_{L/K}(g(x),g(y))=T_{L/K}(x,y)$ for all $g$ in $G$.

In \cite{Bayer-Lenstra} Bayer-Fluckiger and Lenstra prove that for an odd extension of fields, $L/K$, of characteristic not equal to $2$, then $(L,T_{L/K})$ and $(KG,l)$ are isometric as $K$-forms, where $l:K G\times K G\rightarrow K$ is the bilinear extension of $l(g,h)=\delta_{g,h}$ for $g,h\in G$. This is equivalent to the existence of a self-dual normal basis generator for $L$, i.e.,
  an $x\in L$ such that $L=K G.x$ and $T_{L/K}(g(x),h(x))=\delta_{g,h}$.

If $M\subset KG$ is a free $\bo_KG$-lattice, and is self-dual with respect to the restriction of $l$ to $\bo_KG$, then Fainsilber and Morales have proved that if $|G|$ is odd, then $(M,l)\cong (\bo_KG,l)$ (see \cite{Fainsilber+Morales}, Corollary 4.7). The square-root of the inverse different, $A_{L/K}$, is a Galois module that is self-dual with respect to the trace form. From \cite{Erez2}, Theorem 1 we know that $A_{L/K}$ is a free $\bo_K G$-module if and only if $L/K$ is at most weakly ramified, i.e., if the second ramification group is trivial. We know that if $[L:K]$ is odd, then $(L,T_{L/K})\cong(KG,l)$. Therefore, if $[L:K]$ is odd, $(A_{L/K},T_{L/K})$ is isometric to $(\bo_K G,l)$ if and only if $L/K$ is at most weakly ramified. Equivalently, there exists a self-dual integral normal basis generator for $A_{L/K}$ if and only if $L/K$ is weakly ramified.

We remark that this problemma has not been solved in the global setting. Erez and Morales
show in \cite{Erez-Morales} that, for an odd tame abelian extension
of $\Q$, a self-dual integral normal basis does exist for the
square-root of the inverse different. However, in
\cite{Vinatier-Surla}, Vinatier gives an example of a non-abelian
tamely ramified extension, $N/\Q$, where such a basis for $A_{N/\Q}$
does not exist.

We now assume $K$ is a finite unramified extension of $\Q_p$ of
degree $d$. We fix a uniformising parameter, $\pi$, and let
$q=p^d=|k|$. We define $K_{\pi,n}$ to be the unique field obtained
by adjoining to $K$ the $[\pi^n]$-division points of a Lubin-Tate
formal group associated to $\pi$. We note that $K_{\pi,n}/K$ is a
totally ramified abelian extension of degree $q^{n-1}(q-1)$. In
Section \ref{kummer_generators} we choose $\pi=p$ and prove that the $p$th roots of unity are
contained in the field $K_{p,1}$, therefore any abelian extension of exponent $p$ above
$K_{p,1}$ will be a Kummer extension.

Let $\gamma^{p-1}={-p}$. In \cite{Dwork} \S5, Dwork introduces the
exponential power series,
\[E_{\gamma}(X)=\exp(\gamma X -\gamma X^p),\] where the right hand
side is to be thought of as the power series expansion of the
exponential function. In \cite{LangII} Lang presents a proof that
$E_{\gamma}(X)|_{X=\eta}$ converges $p$-adically if $v_p(\eta)\geq 0$ and
also that $E_{\gamma}(X)|_{X=1}$ is equal to a primitive $p$th root
of unity. In Section \ref{kummer_generators} we use Dwork's power series to
construct a set $\{e_0,\ldots,e_{d-1}\}\subset K_{p,1}$ such that $K_{p,2}=K_{p,1}(e_0^{1/p},\ldots,e_{d-1}^{1/p})$. In Section \ref{on_bases} we use these elemmaents to obtain very explicit constructions of self-dual integral normal basis generators for $A_{M/K}$ where $M/K$ is any Galois extension of degree $p$ contained in $K_{p,2}$.

When $K=\Q_p$ and $\pi=p$ the $n$th Lubin-Tate extensions are the
cyclotomic extensions obtained by adjoining $p^n$th roots of unity
to $K$. Hence the study of the Lubin-Tate extensions, $K_{p,n}$, can
be thought of as a generalisation of cyclotomy theory. In
\cite{Erez} Erez studies a weakly ramified $p$-extension of $\Q$
contained in the cyclotomic field $\Q(\zeta_{p^2})$ where
$\zeta_{p^2}$ is a $p^2$th root of unity. He constructs a self-dual
normal basis for the square-root of the inverse different of this
extension. It turns out that the weakly ramified extension studied
by Erez is, in fact, a special case of the extensions studied in
Section \ref{on_bases} and the self-dual normal basis generator that he
constructs is the corresponding basis generator we have generated
using Dwork's power series, so this work generalises results
in \cite{Erez}.


\section{Kummer Generators}\label{kummer_generators}

The construction of abelian Galois extensions of local fields using Lubin-Tate formal groups is standard in local class field theory. For a detailed account see, for example, \cite{iwasawa} or \cite{serre-lubintate}.
We include a brief overview for the convenience of the reader and to fix some notation.

Let $K$ be a finite extension of $\Q_p$. Let $\pi$ be a uniformising
parameter for $\bo_K$ and let $q=|\bo_K/\bp_K|$ be the cardinality
of the residue field. We let $f(X)\in X\bo_K[[X]]$ be such that
\begin{eqnarray} \nonumber f(X)&\equiv&\pi X\mod \deg 2\txt{,\ \ \
\ \ and\ \ \ \ \ \ } f(X)\equiv X^q\mod\pi.\end{eqnarray}

We now let $F_f(X,Y)\in\bo_K[[X,Y]]$ be the unique formal group
which admits $f$ as an endomorphism. This means
$F_f(f(X),f(Y))=f(F_f(X,Y))$ and that $F_f(X,Y)$ satisfies some
identities that correspond to the usual group axioms, see
\cite{serre-lubintate} \S3.2 for full details. For $a\in \bo_K$,
there exists a unique formal power series, $[a]_f(X)\in
X\bo_K[[X]]$, that commutes with $f$ such that $[a]_f(X)\equiv
aX\mod \deg 2$. We can use the formal group, $F_f$, and the formal
power series, $[a]_f$, to define an $\bo_K$-module structure on
$\bp_{\bar{K}}^c=\bigcup_{L}\bp_{L}$, where the union is taken over
all finite Galois extensions $L/K$ where $L\subseteq \bar{K}$. We
are going to look at the $\pi^n$-torsion points of this module. We
let $E_{f,n}=\{x\in\bp_{\bar{K}}^c:[\pi^n]_f(x)=0\} $ and
$K_{\pi,n}=K(E_{f,n})$. We remark that the set $E_{f,n}$ depends on
the choice of the polynomial $f$ but due to a property of the formal
group (see \cite{serre-lubintate} \S3.3 Prop. 4), $K_{\pi,n}$
depends only on the uniformising parameter $\pi$. The extensions
$K_{\pi,n}/K$ are totally ramified abelian extensions.
 If we let $K=\Q_p$ we can let $\pi=p$ and $f(X)=(X+1)^p-1$. We then see that
$K_{p,n}=\Q_p(\zeta_{p^n})$ where $\zeta_{p^n}$ is a primitive
$p^n$th root of unity.

We now let $K$ be an unramified extension of $\Q_p$ of degree $d$.
We note that $q=p^d$ and that we can take $\pi=p$. We can then let
$f(X)=X^q+pX$ and note that $K_{p,1}=K(\beta)$ where
$\beta^{q-1}=-p$. If we let $\gamma=\beta^{(q-1)/(p-1)}$ then
$\gamma^{p-1}=-p$ and $K(\gamma)\subseteq K_{p,1}$. From now on we
will let $K(\gamma)=K'$. We will use Dwork's exponential power
series to construct Kummer generators for $K_{p,2}$ over $K_{p,1}$.

\begin{definition} Let $\gamma^{p-1}=-p$.
We define Dwork's exponential power series as
\[E_{\gamma}(X)=\exp(\gamma X -\gamma X^p),\] where the right hand
side is to be thought of as the power series expansion of the
exponential function.\end{definition}

From \cite{LangII} Chapter 14 \S2, we know that
$E_{\gamma}(X)|_{X=x}$ converges $p$-adically when $v_p(x)\geq 0$
and that $E_{\gamma}(X)\equiv1+\gamma X\mod \gamma^2$. We know then
that $E_{\gamma}(X)|_{X=1}\neq 1$. We now raise Dwork's power series
to the power $p$ and see \[\begin{array}{ll}\exp(\gamma X-\gamma
X^p)^p&=\exp(p(\gamma X-\gamma X^p))\\&= \exp(\gamma pX-\gamma
pX^p)\\&= \exp(\gamma pX)\exp(-\gamma pX^p).\end{array}\]

As $\exp(p\gamma X)|_{X=x}$ converges when
$v_p(x)\geq 0$ we can evaluate both sides at $X=1$ and see
$(\exp(\gamma X-\gamma X^p)^p)|_{X=1} = \exp(\gamma
pX)|_{X=1}\exp(-\gamma pX^p)|_{X=1} =1$. Therefore,
$E_{\gamma}(X)|_{X=1}$ is equal to a primitive $p$th root of unity.
This implies that $K'=K(\gamma)=K(\zeta_p)$.

Let $\zeta_{q-1}$ be a primitive $(q-1)$th root of unity. From
\cite{Frohlich-Taylor} Theorem 25, we know $K$ is uniquely defined and is
equal to $\Q_p(\zeta_{q-1})$. From \cite{Frohlich-Taylor} Theorem 23
we then know that $\bo_K=\mathbb{Z}_p[\zeta_{q-1}]$. We now define
$\{a_i:0\leq i\leq d-1\}$ to be a $\mathbb{Z}_p$-basis for $\bo_K$
where $a_0=1$ and each $a_i$ is a $(q-1)$th root of unity. We also define $e_i=E_{\gamma}(X)|_{X=a_i}$ and let $\K_2=K_{p,1}(e_0^{1/p},e_1^{1/p},\ldots,e_{d-1}^{1/p})$. We will now show that $\K_2=K_{p,2}$.

\begin{lemma} $N_{\K_2/K}(\K_2^*)=<\pi>\times(1+\bp_K^2)$ for some uniformising parameter, $\pi$ of $\bo_K$.\label{K2isKpi2}\end{lemma} \begin{proof} As $E_{\gamma}(X)\equiv 1+\gamma X\mod \gamma^2$ we see that $e_i\equiv1+\gamma a_i\mod\gamma^2$. We define $\mathcal{E}$ to be the set
\[\mathcal{E}=<e_i:0\leq i\leq d-1>(\bo_{K(\gamma)}^{{\times}})^p/(\bo_{K(\gamma)}^{{\times}})^p\]
with multiplicative group structure. We have an isomorphism of
groups $\mathcal{E}\displaystyle\stackrel{\simeq}{\rightarrow}
(\bp_{K})/(p\bp_{K})$, using the additive group structure of
$(\bp_{K})/(p\bp_{K})$, which sends $e_i$ to $a_i$. We remark that
here $p\bp_K=\bp_K^2$. From our selection of the set $\{a_i:0\leq
i\leq d-1\}$ as a basis for $\bo_K$ we know that the $e_i$ must be
linearly independent (multiplicatively) over $\mathbb{F}_p$.
Therefore, we know that $\Gal(\K_2/K_{p,1})$ must be isomorphic to
$\prod_{i=1}^dC_p$. From standard theory (see \cite{serre-lubintate}
\S3), we know $\Gal(K_{p,2}/K_{p,1})\cong\bp_K/\bp_K^2$, which is
also isomorphic to $\prod_{i=1}^dC_p$. Therefore, $Gal(\K_2/K)\cong
Gal(K_{p,2}/K)\cong C_{q-1}{\times}\prod_{i=1}^dC_p$.

The extensions $\K_2/K$ and $K_{p,2}/K$ are both finite abelian
extensions of local fields. By the Artin symbol, (see
\cite{washington} Appendix Theorem 7), we know that
\[K^{\times}/N_{K_{p,2}/K}(K_{p,2}^{\times})\cong\Gal(K_{p,2}/K)\txt{\
\ \  and\ \ \
}K^{\times}/N_{\K_2/K}(\K_2^{\times})\cong\Gal(\K_2/K),\] and so
\[K^{\times}/N_{K_{p,2}/K}(K_{p,2}^{\times})\cong
K^{\times}/N_{\K_2/K}(\K_2^{\times}).\]

From \cite{iwasawa} (Proposition 5.16) we know that
$N_{K_{p,2}/K}(K_{p,2}^{\times})=<p>\times(1+\bp_K^2)$. As
$K^{\times}$ is an abelian group we must then have
$N_{\K_2/K}(\K_2^{\times})\cong<p>\times(1+\bp_K^2)$.

It is straightforward to check that $\K_2/K$ is totally ramified. Therefore, from \cite{Fesenko-Vostokov} IV \S3, we know that
$K^{\times}/N_{\mathcal{K}_2/K}(\mathcal{K_2^{\times}})=\bo_K^{\times}/N_{\mathcal{K}_2/K}(\bo_{\mathcal{K}_2}^{\times})$
($\cong C_{q-1}{\times}\prod_{i=1}^dC_p$). The group
$\bo_K^{\times}\cong C_{q-1}{\times}(1+\bp_K)$, so we know that
\[(1+\bp_K)/N_{\mathcal{K}_2/K}(\bo_{\mathcal{K}_2}^{\times})\cong\prod_{i=1}^dC_p.\]

As $K/\Q_p$ is unramified and $p>2$, the logaritheoremic power series
gives us an isomorphism of groups,
$\log:1+\bp_K\cong\bp_K(\cong\bigoplus_{i=0}^{d-1}\mathbb{Z}_p)$,
using the multiplicative structure of $1+\bp_K$ and the additive
structure of $\bp_K$, see \cite{Fesenko-Vostokov} Chapter IV example
$1.4$ for full details. The maximal $p$-elemmaentary abelian quotient
of $\bigoplus_{i=1}^d\mathbb{Z}_p$ is given by
$\bigoplus_{i=1}^d\mathbb{Z}_p/
\bigoplus_{i=1}^dp\mathbb{Z}_p\cong\prod_{i=1}^dC_p$ and the unique
subgroup that gives this quotient is
$\bigoplus_{i=1}^dp\mathbb{Z}_p$. We then have
$\bp_K/p\bp_K\cong\prod_{i=1}^dC_p$ and using the logaritheoremic
isomorphism we see $(1+\bp_K)/(1+\bp_K)^p\cong\prod_{i=1}^dC_p$.
This means that $(1+\bp_K)^p$ is the unique subgroup of $1+\bp_K$
that gives the maximal $p$-elemmaentary abelian quotient. As above we
have $(1+\bp_K)^p=1+\bp_K^2$ and therefore,
\[N_{\mathcal{K}_2/K}(\bo_{\mathcal{K}_2}^{\times})=1+\bp_K^2.\]

Let $\Pi$ be a uniformising parameter for $\mathcal{K}_2$. As
$\K_2/K$ is totally ramified, $N_{\mathcal{K}_2/K}(\Pi)=\pi$ must be
a uniformising parameter of $K$. Since
$N_{\mathcal{K}_2/K}(\mathcal{K}_2^{\times})$ is a group under
multiplication we know that $<\pi>$ must be a subgroup. We have
already seen that $(1+\bp_K^2)$ is a subgroup, so as
$N_{\mathcal{K}_2/K}(\mathcal{K}^{\times}_2)$ is abelian, we must
have
\[<\pi>{\times}(1+\bp_K^2)\subseteq
N_{\mathcal{K}_2/K}(\mathcal{K}^{\times}_2).\] The subgroups
$<\pi>{\times}(1+\bp_K^2)$ and
$N_{\mathcal{K}_2/K}(\mathcal{K}^{\times}_2)$ both have the same
finite index in $K^{\times}$, therefore we must have equality.\hfill $\Box$
\end{proof}
To prove the next lemmama we will use some properties of the $p$th
Hilbert pairing for a field that contains the $p$th roots of unity.
For full definitions and proofs see \cite{Fesenko-Vostokov} chapter
IV. We include the properties we will need for the convenience of
the reader.

\begin{definition}Let $L$ be a field of characteristic $0$ with fixed
separable algebraic closure $\bar{L}$ and let $\mu_{p}$ be the group
of $p$th roots of unity in $\bar{L}$. Let $\mu_p\subseteq L$. We
define the $p$th Hilbert symbol of $L$ as \[\begin{array}{ll}(\ ,\
)_{p,L}:&L^{{\times}}{\times} L^{{\times}}\longrightarrow
\mu_p\nonumber \\ &\ (a,b)\longmapsto
\frac{(A_L(a))(b^{1/p})}{b^{1/p}} \nonumber, \end{array}\]where
$A_L:L^{{\times}}\longrightarrow \Gal(L^{ab}/L)$ is the Artin map of
$L$ (see \cite{iwasawa} Chapter 6 \S3 for details).\end{definition}

In \cite{Fesenko-Vostokov} Chapter IV, Proposition $5.1$ it is
proved that if $L'/L$ is a finite Galois extension of local fields,
then the Hilbert symbol satisfies the following conditions.
\begin{enumerate}\item$(a,b)_{p,L}=1$ if and only if $a\in N_{L(b^{1/p})/L}(L(b^{1/p})^{\times})$, and $(a,b)_{p,L}=1$ if and only if $b\in
N_{L(a^{1/p})/L}(L(a^{1/p})^{\times})$,
\item $(a,b)_{p,L'}=(N_{L'/L}(a),b)_{p,L}$ for $a\in L'^{{\times}}$ and $b\in
L^{{\times}}$,\item$(a,1-a)_{p,L}=1$  for all $1\neq a\in
L^{{\times}}$,\item$(a,b)_{p,L}=(b,a)^{-1}_{p,L}$.\end{enumerate}

\begin{lemma}\label{pinlemmama} \[p\in N_{\K_2/K}(\K_2^*).\]\end{lemma}\begin{proof} First we show that $(e_i,\zeta_p-1)_{p,K'}=1$ for all $0\leq i\leq d-1$.

Recall that $K'=K(\zeta_p)$ and consider the field extension
$K'/\Q_p(\zeta_p)$. This is an unramified extension of degree $d$.
As $\zeta_p-1\in\Q_p(\zeta_p)$, we can use property 2 of the Hilbert
symbol to show
$(e_i,\zeta_p-1)_{p,K'}=(N_{K'/\Q_p(\zeta_p)}(e_i),\zeta_p-1)_{p,\Q_p(\zeta_p)}.$
Recall that $e_i=E_{\gamma}(X)|_{X=a_i}$ where the set $\{a_i:0\leq
i \leq p-1\}$ forms a basis for $\bo_K$ over $\mathbb{Z}_p$, all the
$a_i$ are $(p^{d}-1)$th roots of unity and $a_0=1$. The action of
the Galois group $\Gal(K/\Q_p)$ on each $a_i$ (which will be the
same as the action of $\Gal(K'/\Q_p(\zeta_p)$) will be generated by
the Frobenius elemmaent, \[\phi_{K/\Q_p}:a_i\mapsto a_i^p.\] We know
that $E_{\gamma}(X)|_{X=x}$ converges when $v_p(x)\geq 0$. As
$a_i^{p^k}\in\bo_{K}^{\times}$, we have that
$E_{\gamma}(X)|_{X=a_i^{p^k}}$ converges for all $k\in\mathbb{Z}$.
Therefore $E_{\gamma}(X^{p^k})|_{X=a_i}$ must converge and
\[\phi_{K/\Q_p}^k(e_i)=E_{\gamma}(X^{p^k})|_{X=a_i},\] where
$\phi_{K/\Q_p}^k$ is the Frobenius elemmaent, $\phi_{K/\Q_p}$, applied
$k$ times. We can now make the following derivation.
\[\begin{array}{ll} N_{K'/\Q_p(\zeta_p)}(e_i)&=\prod_{g\in
\Gal(K'/\Q_p(\zeta_p))}g(e_i)=\prod_{k=0}^{d-1}\phi^k_{K/\Q_p}(e_i)\\
 &=\prod_{k=0}^{d-1}E_{\gamma}(X^{p^k})|_{X=a_i}=\prod_{k=0}^{d-1}\exp(\gamma X^{p^k}-\gamma X^{p^{k+1}})|_{X=a_i}\\
 &=\exp\left((\gamma X-\gamma X^p)+(\gamma X^p-\gamma X^{p^2})+\ldots+(\gamma X^{p^{d-1}}-X^{p^d})\right)|_{X=a_i}\\
&=  \exp(\gamma X-\gamma X^{p^d})|_{X=a_i}.\end{array}\]

We now consider raising to the power $p$ and see
\[\begin{array}{ll} \exp(\gamma X-\gamma X^{p^d})^p&=\exp(p(\gamma X-\gamma X^{p^d}))\\
&=\exp(p\gamma X-p\gamma X^{p^d})\\  &=\exp(p\gamma X)\exp(-p\gamma
 X^{p^d}).\end{array}\]

The power series $\exp(p\gamma X)|_{X=x}$ will converge when
$v_p(x)\geq 0$ so we can evaluate at $X=a_i$ and see,
$(N_{K'/\Q_p(\zeta_p)}(e_i))^p=1$. Therefore
$N_{K'/\Q_p(\zeta_p)}(e_i)$ is a $p$th root of unity for all $0\leq
i\leq d-1$. If $N_{K'/\Q_p(\zeta_p)}(e_i)=1$ then
$(N_{K'/\Q_p(\zeta_p)}(e_i),1-\zeta_p)_{p,\Q_p(\zeta_p)}=(1,1-\zeta_p)_{p,\Q_p(\zeta_p)}=1$, so we now assume
$N_{K'/\Q_p(\zeta_p)}(e_i)$ is a primitive $p$th root of unity. From
property 3 of the Hilbert symbol we know that
$(\zeta_p,1-\zeta_p)_{p,\Q_p(\zeta_p)}=1$. We know that for $1\leq
k\leq p-1$ then
$\Q_p(\zeta_p)(\zeta_p^{1/p})=\Q_p(\zeta_p)(\zeta_p^{k/p})$, and so
from property $1$ of the Hilbert symbol we know that
$(\zeta_p^k,1-\zeta_p)_{p,\Q_p(\zeta_p)}=1$. This means that
$(e_i,1-\zeta_p)_{p,K'}=1$ for all $0\leq i\leq d-1$. We now let
$\xi_i\in K'(e_i^{1/p})$ be such that
$N_{K'(e_i^{1/p})/K'}(\xi_i)=1-\zeta_p$. As $p$ is odd,
$N_{K'(e_i^{1/p})/K'}(-\xi_i)=\zeta_p-1$, and therefore
\[(e_i,\zeta_p-1)_{p,K'}=1\] for all $0\leq i\leq d-1$.

Next we show that $\zeta_p-1\in N_{\K_2/K'}(\K_2^{{\times}})$. We
have just shown that $\zeta_p-1\in
N_{K'(e_0^{1/p})/K'}(K'(e_0^{1/p})^{{\times}})$. We assume, for
induction, that \[\zeta_p-1\in N_{K'(e_0^{1/p},\ldots
e_j^{1/p})/K'}(K'(e_0^{1/p},\ldots e_j^{1/p})^{{\times}})\] for some
$0\leq j\leq p-1$. Let $\eta\in
K'(e_0^{1/p},\ldots,e_j^{1/p})^{{\times}}$ be such that
$N_{K'(e_0^{1/p},\ldots e_j^{1/p})/K'}(\eta)=\zeta_p-1$. As
$e_{j+1}\in K'$ we can make the following derivation:
\[\begin{array}{ll}(\eta,e_{j+1})_{p,K'(e_0^{1/p},\ldots, e_j^{1/p})}&=(N_{K'(e_0^{1/p},\ldots, e_j^{1/p})/K'}(\eta),e_{j+1})_{p,K'}
\\ &=(\zeta_p-1,e_{j+1})_{p,K'}\\  &=(e_{j+1},\zeta_p-1)_{p,K'}^{-1}=1.\end{array}\]

Therefore, \[\eta\in
N_{K'(e_0^{1/p},\ldots,e_{j+1}^{1/p})/K'(e_0^{1/p},\ldots,e_j^{1/p})}(K'(e_0^{1/p},\ldots,e_{j+1}^{1/p})^{{\times}}),\]
and so \[(\zeta_p-1)\in
N_{K'(e_0^{1/p},\ldots,e_{j+1}^{1/p})/K'}(K'(e_0^{1/p},\ldots,{e_{j+1}^{1/p}})^{{\times}}).\]
By induction on $j$ we see that $(\zeta_p-1)\in
N_{\K_2/K'}(\K_2^{{\times}})$.

Finally we note that the minimal polynomial of $\zeta_p-1$ over $K$
is $f(X)=((X+1)^p-1)/X$. The constant term in $f(X)$ is equal to $p$
and $K'$ is the splitting field of $f(X)$. Therefore, as $[K':K]$ is
even, $N_{K'/K}(\zeta_p-1)=p$. The norm map is transitive, so we know
that $p\in N_{\K_2/K}(\K_2^{{\times}})$.\hfill $\Box$
\end{proof}

\begin{theorem}\[K_{p,2}=K_{p,1}(e_0^{1/p},e_1^{1/p},\ldots,e_{d-1}^{1/p}).\]\end{theorem}
\begin{proof} From Lemma \ref{K2isKpi2} we know that
$N_{\K_2/K}(\K_2^{\times})=<\pi>\times1+\bp_K^2$ where $\pi=up$ for
some $u\in\bo_K^{\times}$. From Lemma \ref{pinlemmama} we know that
$p\in N_{\K_2/K}(\K_2^{\times})$ and therefore that
$N_{\K_2/K}(\K_2^{\times})=<p>\times1+\bp_K^2$. From \cite{iwasawa}
Proposition 5.16, we know that
$N_{K_{p,2}/K}(K_{p,2}^{\times})=<p>\times(1+\bp_K^2)$. As $\K_2/K$
and $K_{p,2}/K$ are both finite abelian extensions of local fields
contained in $\bar{K}$ and
$N_{K_{p,2}/K}(K_{p,2}^{\times})=N_{\K_2/K}(\K_2^{\times})$, from
\cite{washington} Appendix Theorem 9, we know that
$\K_2=K_{p,2}$.\hfill $\Box$\end{proof}

\section{Explicit Self-Dual Normal Bases for $A_{M/K}$} \label{on_bases}

We begin this section by describing the intermediate fields of $K_{p,2}/K$ that we are going to study.
The extension $K_{p,2}/K_{p,1}$ is a totally ramified abelian extension of degree $q$.
There will be $(q-1)/(p-1)$ intermediate fields, $N_j$ such that $[K_{p,2}:N_j]=q/p$ and $[N_j:K_{p,1}]=p$.
The $p$th roots of unity are contained in $K_{p,1}$, so for each $j$, the extension $N_j/K_{p,1}$ will be a Kummer extension.
 We recall that $\{a_i:0\leq i\leq d-1\}$ is a $\mathbb{Z}_p$-basis for $\bo_K$ where $a_0=1$ and all
 the $a_i$ are $(q-1)$th roots of unity. We have shown that $K_{p,2}=K(e_0^{1/p},e_1^{1/p},\ldots e_{d-1}^{1/p})$,
  where the $e_i=E_{\gamma}(X)|_{X=a_i}$. Therefore each $N_j=K_{p,1}(x_j^{1/p})$ for
  $x_j=\prod_{i=0}^{d-1}e_i^{n_i}$ for some $0\leq n_i\leq p-1$, not all zero.
  We now note that for all $x=\prod_{i=0}^{d-1}e_i^{n_i}$ as above, we have $x\in K'(=K(\gamma)=K(\zeta_p))$.
  Therefore $K'(x_j^{1/p})$ is the unique extension of $K'$ of degree $p$ contained in $N_j$.
   There is also a unique extension of $K$ of degree $p$ contained in $N_j$, we shall call this extension $M_j$
   and let $\Gal(K'(x_j^{1/p})/M_j)=\Delta_j$. From now on we will drop the subscript for $N_j$, $x_j$,
   $M_j$ and $\Delta_j$ as the following results do not depend on which $x_j=\prod_{i=0}^{d-1}e_i^{n_i}$ we pick. To clarify, we will describe these extensions in Fig. \ref{abextofK}.

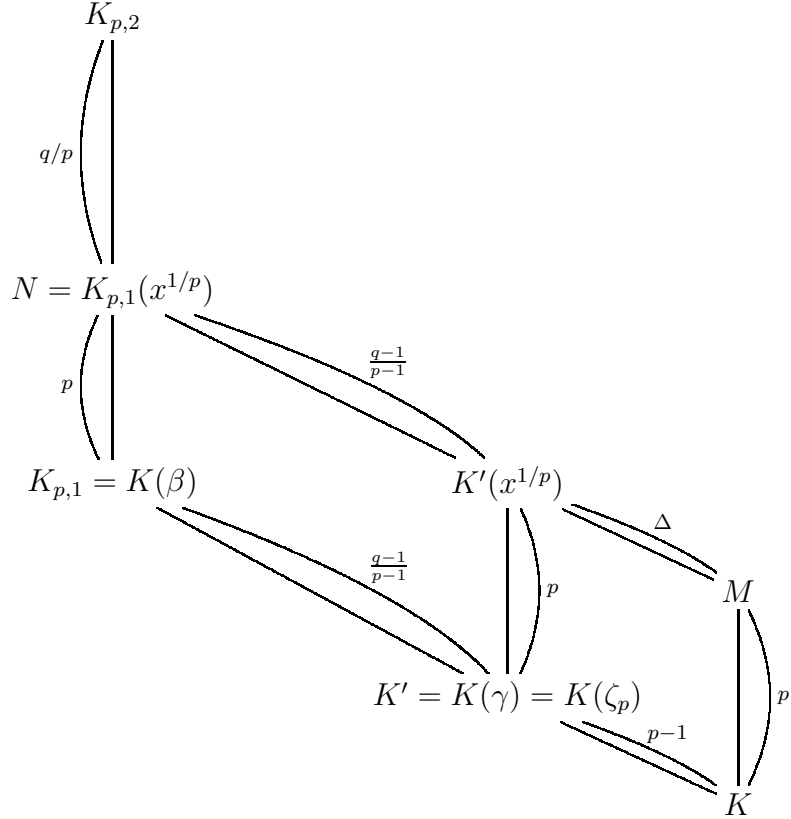
\begin{figure}[ht]\[\xymatrix{K_{p,2}\ar@{-}@/l0pc/[ddd]\ar@{-}@/l1pc/[ddd]_{q/p}\\&\\&\\N=K_{p,1}(x^{1/p}
)\ar@{-}@/l0pc/[dd]\ar@{-}@/l1pc/[dd]_{p}\ar@{-}@/l0pc/[ddrr]\ar@{-}@/r2pc/[ddrr]^{\frac{q-1}{p-1}}\\&\\K_{p,1}=K(\beta)
\ar@{-}@/r2pc/[ddrr]^{\frac{q-1}{p-1}}\ar@{-}@/l0pc/[ddrr]&&K'(x^{1/p})\ar@{-}@/r1pc/[dd]^{p}\ar@{-}@/r1pc/[dr]^{\Delta}\ar@{-}@/l0pc/[dr]\ar@{-}@/l0pc/[dd]
\\&&&M\ar@{-}@/r1pc/[dd]^{p}\ar@{-}@/l0pc/[dd]\\&&K'=K(\gamma)=K(\zeta_p)\ar@{-}@/r1pc/[dr]^{p-1}\ar@{-}@/l0pc/[dr]\\&&&K\\}\]\caption{Abelian extensions of $K$}\label{abextofK}\end{figure}

We also let $\Gal(K'(x^{1/p})/K')=G$, and as all the groups we are
dealing with are abelian we will use an abuse of notation and write
$\Gal(M/K)=G$ and $\Gal(K'/K)=\Delta$.

Let $A_{M/K}=\D_{M/K}^{-1/2}$ be the square-root of the inverse
different of $M/K$. The aim now is to show that
$(1+Tr_{\Delta}(x^{1/p}))/p$ is a self-dual normal basis for
$A_{M/K}$.

We remark that if $K=\Q_p$, then $K'=K_{p,1}$,
$N_1=K_{p,2}=K'(x^{1/p})$ and the only choice for $x$ is
$E_{\gamma}(X)|_{X=1}=\zeta_p$. In \cite{Erez} Erez shows that in
this case $(1+Tr_{\Delta}(\zeta_p^{1/p}))/p$ does indeed give a
self-dual normal basis for $A_{M/K}$. So the situation we describe
generalises the work in \cite{Erez}.

Before we proceed to the main results of this section we must make some basic calculations about the field extensions to be studied.
\begin{lemma}\[v_M(A_{M/K})=1-p.\]\end{lemma}
\begin{proof}
We first calculate the ramification groups of $K_{p,2}/K_{p,1}$. We
recall that $f(X)=X^q+pX$. If we let $u\in \mu_{q-1}\cup\{0\}(=
k)$, clearly $[u](X)=uX$ and $[up](X)=u[p](X)$. Let $\alpha$ be a
primitive $[p^2]$-division point for $F_f(X,Y)$. We see that
 \[\begin{array}{ll}f([up+1](\alpha))&=f(F(u[p](\alpha),\alpha))\\ &=F(f(u[p](\alpha)),f(\alpha))\\  &=F(uf^2(\alpha),f(\alpha))\\
 &=f(\alpha).\end{array}\]
 Therefore $[up+1](\alpha)$ is another primitive $[p^2]$-division point and the Galois conjugates of $\alpha$ over $K_{p,1}$ are given by $[up+1](\alpha)$ for $u\in \mu_{q-1}\cup\{0\}$.

Given $f(X)\in\bo_K[X]$ such that $f(X)\equiv pX\mod \deg 2$ and
$f(X)\equiv X^q\mod p$, the standard proof in the literature of the
existence of a formal group $F(X,Y)\in \bo_K[[X,Y]]$ such that $F$ commutes with
$f$ uses an iterative process for calculating $F_f$. See, for
example, \cite{serre-lubintate} \S3.5 Proposition 5 or
\cite{iwasawa} III, Proposition 3.12. The $i$th iteration calculates
$F(X,Y)\mod\deg(i+1)$ and passage to the inductive limit gives
$F(X,Y)$. We will use this process to calculate the first few terms
of $F(X,Y)$.

We will let $F^i(X,Y)\equiv F(X,Y)\mod \deg (i+1)$ and define $E_i$
to be the $i$th error term, i.e.,
$E_i=f(F^{i-1}(X,Y))-F^{i-1}(f(X),f(Y))\mod\deg(i+1)$. From \cite{serre-lubintate} \S3.5
Proposition 5 we then have \[F^{i+1}(X,Y)=F^i(X,Y)-\frac{E_i}{p(1-p^{i-1})}.\]

$F(X,Y)$ is a formal group, so $F^1(X,Y)=X+Y$. We then see
\[\begin{array}{ll}f(F^1(X,Y))-F^1(f(X),f(Y))&=(X+Y)^q+p(X+Y)-(X^q+pX+Y^q+pY)\\&=\sum\limits_{i=1}^{q-1}\binom{q}{i}X^iY^{q-i}.\end{array}\] So the error terms will be
$E_i=0$ for $2\leq i\leq q-1$ and $E_q=\sum\limits_{i=1}^{q-1}\binom{q}{i}X^iY^{q-i}$. From \cite{serre-lubintate} \S3.5
Proposition 5, we then get
\[F(X,Y)\equiv X+Y-\frac{\sum\limits_{i=1}^{q-1}\binom{q}{i}X^iY^{q-i}}{p(1-p^{q-1})}\mod \deg(q+1).\]  We now substitute $X=\alpha$ and  $Y=u[p](X)=u(\alpha^{q}+p\alpha)$ into our expression for $F(X,Y)$ and
see that
\[\begin{array}{ll} [1+up](\alpha)&\equiv \alpha+u(\alpha^{q}+p\alpha)-\frac{\sum\limits_{i=1}^{q-1}\binom{q}{i}\alpha^i(u(\alpha^{q}+p\alpha))^{q-i}}{p(1-p^{q-1})}\mod
\alpha^{q+1}\\&\equiv (1+up)\alpha+\left(u-\frac{\sum\limits_{i=1}^{q-1}(up)^{q-i}\binom{q}{i}}{p(1-p^{q-1})}\right)\alpha^{q}\mod
\alpha^{q+1}.\end{array}\]

Let $\Gamma=Gal(K_{p,2}/K_{p,1})$. We know that $\alpha$ is a
uniformising parameter for $\bo_{K_{p,2}}$ and that
$p\in\bp_{K_{p,2}}^{q(q-1)}$. An elemmaent $s\in \Gamma$ is in the
$i$th ramification group (with the lower numbering), $\Gamma_i$, if
and only if $s(\alpha)/\alpha\equiv 1 \mod \bp_{K_{p,2}}^i$, see
\cite{Serre} IV \S2 Prop 5. We have shown that for $1\neq s\in
\Gamma$ then $s(\alpha)/\alpha\equiv
1+u\alpha^{q-1}\mod\bp_{K_{p,2}}^q$. Therefore, $\Gamma=\Gamma_i$
for $0\leq i\leq (q-1)$ and $\Gamma_{q}=\{1\}$.

To calculate the ramification groups of $N/K_{p,1}$ we need to
change the numbering of the ramification groups of $K_{p,2}/K_{p,1}$ from lower
numbering to upper numbering. From \cite{Serre} IV \S3 we have
$\Gamma^{-1}=\Gamma$, $\Gamma^0=\Gamma_0$ and
$\Gamma^{\phi(m)}=\Gamma_m$ where
$\phi(m)=\frac{1}{|\Gamma_0|}\sum\limits_{i=1}^m|\Gamma_i|$. A
straightforward calculation then shows that the upper numbering is
actually the same as the lower numbering. From \cite{Serre} IV \S3
Proposition 14 we then know that $Gal(N/K_{p,1})=Gal(N/K_{p,1})^i$
for $0\leq i \leq(q-1)$. and $Gal(N/K_{p,1})^q=\{1\}$ and switching
back to the lower numbering we have
$Gal(N/K_{p,1})=Gal(N/K_{p,1})_i$ for $0\leq i \leq(q-1)$. and
$Gal(N/K_{p,1})_q=\{1\}.$

From \cite{Serre} IV \S2 Proposition 4, we have the formula,
\[v_N({\frak D}_{N/K_{p,1}}) = \sum_{i\geqslant
0}(|Gal(N/K_{p,1})_i|-1),\] and so $v_N(\D_{N/K_{p,2}})=q(p-1)$. The
extensions $N/M$ and $K_{p,1}/K$ are both totally, tamely ramified
extensions of degree $q-1$, so from the formula above we know that
$v_{N}(\D_{N/M})=v_{K_{p,1}}(\D_{K_{p,1}/K})=q-2$. From
\cite{Frohlich-Taylor} III.2.15 we know, for a separable tower of
fields $L''\supseteq L'\supseteq L$, the differents of these field
extensions are linked by the formula
$\D_{L''/L}=\D_{L''/L'}\D_{L'/L}$. We therefore have
$v_M(\D_{M/K})=2(p-1)$, and so $v_M(A_{M/K})=1-p$.\hfill $\Box$
\end{proof}

\begin{remark}We remark that this lemmama implies that $M/K$ is weakly ramified.
\end{remark}

We now prove a very useful result that makes finding self-dual integral normal bases much easier.

\begin{lemma}Let $a$ be an elemmaent of $A_{L/K}$ that is
self-dual with respect to the trace form, (i.e.,
$T_{L/K}(g(a),h(a))=\delta_{g,h}$ for all $g,h\in G$),
 then $A_{L/K}=\bo_K[G].a$.\label{in and self-dual}
\end{lemma}
\begin{proof}
Let $a\in A_{L/K}$ be as given. The square-root of the inverse different, $A_{L/K}$, is a fractional
$\bo_L$-ideal stable under the action of the Galois group, $G$,
therefore $\bo_K [G].a\subseteq A_{L/K}$.

The inclusion of $\bo_K$-lattices, $\bo_K [G].a\subseteq A_{L/K}$,
means that $A_{L/K}^D\subseteq(\bo_K[G].a)^D$ where $D$ denotes the
$\bo_K$-dual taken with respect to the trace form. As $A_{L/K}=A_{L/K}^D$,
we have $A_{L/K}\subseteq(\bo_K[G].a)^D$. We know that $\bo_K[G].a$
is $\bo_K$-free on the basis $\{g(a):g\in G\}$, so $(\bo_K[G].a)^D$
is $\bo_K$-free on the dual basis with respect to the trace form,
which is $\{g(a):g\in G\}$. Therefore $(\bo_K[G].a)^D=\bo_K[G].a$
and $A_{L/K}\subseteq\bo_K[G].a$, and so $A_{L/K}=\bo_K[G].a$.\hfill $\Box$
\end{proof}

For each $x=\prod_{i=0}^{d-1}e_i^{n_i}$ with $0\leq n_i\leq p-1$ not all zero, we know that there exists $u\in\bo_K^{\times}$ such that $x\equiv
1+u\gamma\mod\gamma^2$. The elemmaent
$\gamma$ is a uniformising parameter for $\bo_{K'}$, therefore,
$x\in\bo_{K'}^{\times}$ and $x-1$ will also be a uniformising
parameter for $\bo_{K'}$. Using the binomial theorem we note that
$(x^{1/p}-1)^p=x-1+py$ where $v_{K'(x^{1/p})}(y)\geq 0$. Therefore
$v_{K'(x^{1/p})}((x^{1/p}-1)^p)=p$ and
$v_{K'(x^{1/p})}(x^{1/p}-1)=1$, so $x^{1/p}-1$ is a uniformising
parameter for $\bo_{K'(x^{1/p})}$.

\begin{lemma}\[\frac{1+Tr_{\Delta}(x^{1/p})}{p}\in A_{M/K}.\]\label{inlemmama}\end{lemma}
\begin{proof} We have just shown that $x^{1/p}-1$
is a uniformising parameter for $\bo_{K'(x^{1/p})}$. As
$K'(x^{1/p})/M$ is a totally, tamely ramified extension, we know
that $Tr_{\Delta}(x^{1/p}-1)\in\bp_{M}$ so
$v_{M}(Tr_{\Delta}(x^{1/p}-1))\geq1.$ We know that
\[Tr_{\Delta}(x^{1/p}-1)=Tr_{\Delta}(x^{1/p})-(p-1)=(1+Tr_{\Delta}(x^{1/p}))-p.\]
Therefore, $v_{M}(1+Tr_{\Delta}(x^{1/p}))\geq 1$ and
$v_{M}\left(\frac{1+Tr_{\Delta}(x^{1/p})}{p}\right)\geq 1-p.$ Since
$v_{M}(A_{M/K})=1-p$, we must have
$\frac{1+Tr_{\Delta}(x^{1/p})}{p}\in A_{M/K}$.\hfill $\Box$\end{proof}

\begin{lemma}Let $x=\prod_{i=0}^{d-1}e_i^{n_i}$ for some $n_i\in\mathbb{Z}^+$, and let $\delta\in\Delta=\Gal(K'(x^{1/p})/M)$.
Let $\delta:\gamma\mapsto\chi(\delta)\gamma$ with
$\chi(\delta)\in\mu_{p-1}$, then
${\delta}(x)=x^{\chi(\delta)}.$\label{powerlemmama}
\end{lemma}
\begin{proof}
As $\chi(\delta)^p=\chi(\delta)$, for all $\delta\in\Delta$ we have
the following equality:

\[\exp(\chi(\delta)\gamma X-\chi(\delta)\gamma X^p)=\exp\left((\chi(\delta)\gamma X)+\frac{(\chi(\delta)\gamma
X)^p}{p}\right).\]

As $\chi(\delta)$ is a unit we know, from \cite{LangII} Chapter 14
\S2 that $\exp\left((\chi(\delta)\gamma X)+\frac{(\chi(\delta)\gamma
X)^p}{p}\right)|_{X=y}$ will converge when $v_p(y)\geq 0$.
Therefore, $\exp(\chi(\delta)\gamma X-\chi(\delta)\gamma
X^p)|_{X=a_i}$ will converge. We can now make the following
derivation:
\[\begin{array}{ll}(E_{\gamma}(X)|_{X=a_i})^{\chi(\delta)}&=(\exp(\gamma X-\gamma X^p)|_{X=a_i})^{\chi(\delta)}\\
&=\exp(\chi(\delta)(\gamma X-\gamma X^p))|_{X=a_i}\\
&=\exp(\chi(\delta)\gamma X-\chi(\delta)\gamma X^p)|_{X=a_i}.
\end{array}\]

As $a_i$ is fixed by all $\delta\in\Delta$ we see that
\[\delta(\gamma X-\gamma
X^p)|_{X=a_i}=(\delta(\gamma)X-\delta(\gamma)X^p)|_{X=a_i}=(\chi(\delta)\gamma
X-\chi(\delta)\gamma X^p)|_{X=a_i}.\]

 As $\exp(\chi(\delta)\gamma
X-\chi(\delta)\gamma X^p)|_{X=a_i}$ converges we must then have
\[\begin{array}{ll} \exp(\chi(\delta)\gamma X-\chi(\delta)\gamma X^p)|_{X=a_i}&=\exp(\delta(\gamma) X-\delta(\gamma) X^p)|_{X=a_i}
\\&={\delta}(\exp(\gamma X-\gamma X^p)|_{X=a_i}) \\&={\delta}(E_{\gamma}(X)|_{X=a_i}).
\end{array}\]

Therefore, ${\delta}(e_i)=(e_i)^{\chi(\delta)}$ for all $0\leq i\leq
(d-1)$, which means $\delta(x)=x^{\chi(\delta)}$.\hfill $\Box$
\end{proof}
\begin{lemma}\label{selfdual}Let $g\in Gal(M/K)$, then \[T_{M/K}\left(\frac{1+Tr_{\Delta}(x^{1/p})}{p},g\left(\frac{1+Tr_{\Delta}(x^{i/p})}{p}\right)\right)=\delta_{1,g}.\]\end{lemma}
\begin{proof}First we observe that $Tr_{G}(x^{i/p})=\sum_{g\in G}g(x^{i/p})=x^{1/p}\sum_{j=0}^{p-1}\zeta_p^{ij}=0$ for all $p\slashed{|}i$.
The trace map is transitive, so
$Tr_{{G}}(Tr_{\Delta}(x^{i/p}))=Tr_{{\Delta}}(Tr_G(x^{i/p}))=Tr_{{\Delta}}(0)=0$ for $p\slashed{|}i$.
We make the following derivation:
\[\begin{array}{ll}
Tr_{{G}}\left(\left(\frac{1+Tr_{\Delta}(x^{1/p})}{p}\right)g\left(\frac{1+Tr_{\Delta}(x^{1/p})}{p}\right)\right)
&=Tr_{{G}}\left(\left(\frac{1+Tr_{\Delta}(x^{1/p})}{p}\right)\left(\frac{1+g(Tr_{\Delta}(x^{1/p}))}{p}\right)\right)
\\
&=Tr_{{G}}\left(\frac{1+Tr_{\Delta}(x^{1/p})+g(Tr_{\Delta}(x^{1/p}))+Tr_{\Delta}(x^{1/p})g(Tr_{\Delta}(x^{1/p}))}{p^2}\right)\\
 &=Tr_{{G}}\left(\frac{1+Tr_{\Delta}(x^{1/p})g(Tr_{\Delta}(x^{1/p}))}{p^2}\right)\\
&=\frac{p+Tr_{{G}}(Tr_{\Delta}(x^{1/p})g(Tr_{\Delta}(x^{1/p})))}{p^2}.
\end{array}\]
The right-hand side of this equation equals $1$ if and only if
$Tr_{{G}}(Tr_{\Delta}(x^{1/p})g(Tr_{\Delta}(x^{1/p})))=(p-1)p$,
and it equals $0$ if and only if
$Tr_{{G}}(Tr_{\Delta}(x^{1/p})g(Tr_{\Delta}(x^{1/p})))=-p$.
Therefore it is sufficient to show
\[Tr_{{G}}(Tr_{\Delta}(x^{1/p})g(Tr_{\Delta}(x^{1/p})))=\left\{
\begin{array}{cl}(p-1)p & \txt{if
$g=id$} \\ -p & \txt{if $g\neq id$}.\end{array}\right. \]
From Lemma \ref{powerlemmama} we know that
$\delta(x)=x^{\chi(\delta)}$. This means that
$\delta(x^{1/p})=\zeta_{\delta}x^{\chi(\delta)/p}$ for some
$\zeta_{\delta}\in\mu_p$. We know that
$\mu_{p-1}\subset\mathbb{Z}_p^{\times}$ so we can write
$\chi(\delta)\equiv j(\delta)\mod p$, for some $1\leq j(\delta)\leq
(p-1)$ and note that $j(\delta)=j(\delta')$ if and only if $\delta=\delta'$. We can therefore define a set of constants
$\{\lambda_{j(\delta)}\in\bo_{K'}:\delta\in\Delta\}$ such that
$\delta(x^{1/p})=\lambda_{j(\delta)}x^{j(\delta)/p}$. We now define
$\sigma\in\Delta$ to be the involution such that $\chi(\sigma)=-1$ and $j(\sigma)=p-1$
and note that $\sigma(\zeta_p)=\zeta_p^{-1}$. We consider the double
action of $\sigma$ on $x^{1/p}$. We have
$\sigma(x^{1/p})=\zeta_{\sigma} x^{\chi(\sigma)/p}=\zeta_{\sigma} x^{-1/p}$, so
\[\begin{array}{ll}\sigma^2(x^{1/p})&=\sigma(\zeta_{\sigma})\sigma(x^{-1/p})\\ &=\zeta_{\sigma}^{-1}\sigma(x^{1/p})^{-1}\\  &=\zeta_{\sigma}^{-1}(\zeta_{\sigma}x^{-1/p})^{-1}\\
 &=\zeta_{\sigma}^{-2}x^{1/p}.\end{array}\]
As $\sigma$ is an involution, $x^{1/p}=\zeta_{\sigma}^{-2}x^{1/p}$,
so we have $\zeta_{\sigma}=1$. Therefore, $\sigma(x^{1/p})=x^{-1/p}=(1/x)x^{(p-1)/p}$, and so $\lambda_{p-1}=1/x$.

For $g\in G$ we know that $g(x^{1/p})=\zeta^ix^{1/p}$ for some
$0\leq i \leq p-1$ with $i=0$ when $g=id$. Using this notation we
make the following derivation:
\[\begin{array}{ll}
Tr_{{G}}(Tr_{\Delta}(x^{1/p})g(Tr_{\Delta}(x^{1/p}))) &=
Tr_{{G}}\left(\left(\sum\limits_{\xi\in\Delta}{\xi}(x^{1/p})\right)\left(g(\sum_{\eta\in\Delta}\eta (x^{1/p}))\right)\right)\\
&=
Tr_{{G}}\left(\sum_{\xi\in\Delta}\sum_{\eta\in\Delta}{\xi}(x^{1/p})g(\eta
(x^{1/p}))\right)\\ &=
Tr_{{G}}\left(\sum_{\xi\in\Delta}\sum_{\eta\in\Delta}{\xi}(x^{1/p})\eta
g(x^{1/p})\right)\txt{\ \ \ as ${G}\times{\Delta}$ is abelian}\\
&=
Tr_{{G}}\left(\sum_{\xi\in\Delta}\sum_{\delta\in\Delta}{\xi}(x^{1/p})\xi\delta
g(x^{1/p})\right)\txt{\ \ \ where $\delta=\xi^{-1}\eta$}\\
 &=
Tr_{{G}}\left(\sum\limits_{\xi\in\Delta}{\xi}\left(\sum\limits_{\delta\in\Delta}(x^{1/p}){\delta
g}(x^{1/p})\right)\right) \\ &=
Tr_{{G}\times{\Delta}}\left(\sum\limits_{\delta\in\Delta}(x^{1/p}){\delta
g}(x^{1/p})\right)\\  &=
\sum\limits_{\delta\in\Delta}Tr_{{G}\times{\Delta}}\left((x^{1/p}){\delta
g}(x^{1/p})\right)\\
 &=\sum\limits_{\delta\in\Delta}Tr_{{G}\times{\Delta}}\left((x^{1/p}){\delta
}(\zeta_p^i(x^{1/p}))\right)\\&=
\sum\limits_{\delta\in\Delta}Tr_{{G}\times{\Delta}}\left((x^{1/p}){\delta}(x^{1/p}){\delta
}(\zeta_p^i)\right)\\&=\sum\limits_{\delta\in\Delta}Tr_{{G}\times{\Delta}}\left((x^{1/p})(\lambda_{j(\delta)}x^{j(\delta)/p})\zeta_p^{ij(\delta)}\right)\\
&=\sum\limits_{j=1}^{p-1}Tr_{{G}\times{\Delta}}\left((x^{1/p})(\lambda_jx^{j/p})\zeta_p^{ij}\right)\\
&=\sum\limits_{j=1}^{p-1}Tr_{{G}\times{\Delta}}\left((x^{(j+1)/p})\lambda_j\zeta_p^{ij}\right).
\end{array}\]

Now
$Tr_{{G}\times{\Delta}}((x^{(j+1)/p})\lambda_j\zeta_p^{ij})=Tr_{\Delta}(\lambda_j\zeta_p^{ij}(Tr_G(x^{(j+1)/p})))$
 as $\lambda_j,\zeta_p^{ij}\in K'$ and we saw above that $Tr_G(x^{(j+1)/p})=0$ apart from when
$j=p-1$.  Using this and that fact that $\lambda_{p-1}=1/x$ we see that
\begin{eqnarray}Tr_G(Tr_{\Delta}(x^{1/p})g(Tr_{\Delta}(x^{1/p})))&=&Tr_{\Delta}((1/x)\zeta_p^{i(p-1)}(Tr_G(x))) \nonumber \\ \nonumber&=&pTr_{\Delta}(\zeta^{-i}).\end{eqnarray} Therefore, \[Tr_{{G}}(Tr_{\Delta}(x^{1/p})g(Tr_{\Delta}(x^{1/p})))=\left\{
\begin{array}{cl}(p-1)p & if
g=id \\ -p & if g\neq id\end{array}\right. \] as required.\hfill $\Box$
\end{proof}

\begin{theorem}For all $x_j=\prod_{i=0}^{d-1}e_i^{n_i}$ with $0\leq n_i\leq p-1$ not all zero, \[\frac{1+Tr_{\Delta_j}(x_j^{1/p})}{p}\] is a self-dual normal basis generator for $A_{M_j/K}$. \end{theorem}
\begin{proof}From Lemma \ref{inlemmama} we know that $(1+Tr_{\Delta_j}(x_j^{1/p}))/p\in A_{M/K}$. From Lemma \ref{selfdual} we know that \[T_{M/K}\left(\frac{1+Tr_{\Delta_j}(x_j^{1/p})}{p},g\left(\frac{1+Tr_{\Delta_j}(x_j^{1/p})}{p}\right)\right)=\delta_{1,g}\] for all $g\in Gal(M/K)$. Therefore, using Lemma \ref{in and self-dual} we know that $(1+Tr_{\Delta_j}(x_j^{1/p}))/p$ is a self-dual normal basis generator for $A_{M_j/K}$.\hfill $\Box$\end{proof}

\begin{remark}\

\begin{enumerate} \item We remark that for every Galois extension, $M'/K$, of degree $p$ contained in $K_{p,2}$
we can construct a self-dual normal basis generator for $A_{M'/K}$ in this way.

\item Let $\mathcal{M}=\prod_{j}M_j$ be the compositum of the field extensions $M_j$
for all $j$ ($\mathcal{M}$ is actually equal to
$\prod_{x_j\in\{e_i:0\leq i\leq d-1\}}M_i$).  This is a weakly
ramified extension of $K$ of degree $q$. The product
$\prod_{i=0}^{q-1}(1+Tr_{\Delta}(e_i^{1/p}))/(p)$ is then a
self-dual elemmaent in $\mathcal{M}$ and seems like the obvious choice
for a self-dual integral normal basis generator for
$A_{\mathcal{M}/K}$. However
$v_{\mathcal{M}}(A_{\mathcal{M}/K})=1-q$, and so
$\prod_{i=0}^{q-1}(1+Tr_{\Delta}(e_i^{1/p}))/(p)\slashed{\in}A_{\mathcal{M}/K}$
so generalisation up to $\mathcal{M}$ is not as straight forward as
one might hope.\end{enumerate}\end{remark}

\bibliography{bib}



\end{document}